\documentclass[11pt]{article}

\usepackage[T1]{fontenc}
\usepackage[utf8]{inputenc}
\usepackage[margin=1in]{geometry}
\usepackage{amsmath,amssymb,amsthm,mathtools}
\usepackage{microtype}
\usepackage{enumitem}
\usepackage[colorlinks=true,citecolor=blue,linkcolor=blue,urlcolor=blue]{hyperref}
\usepackage[nameinlink,capitalize,noabbrev]{cleveref}

\newtheorem{theorem}{Theorem}[section]
\newtheorem{lemma}[theorem]{Lemma}
\newtheorem{proposition}[theorem]{Proposition}
\newtheorem{corollary}[theorem]{Corollary}
\theoremstyle{remark}
\newtheorem{remark}[theorem]{Remark}

\crefname{theorem}{Theorem}{Theorems}
\crefname{lemma}{Lemma}{Lemmas}
\crefname{proposition}{Proposition}{Propositions}
\crefname{corollary}{Corollary}{Corollaries}
\crefname{remark}{Remark}{Remarks}

\newcommand{\C}{\mathbb C}
\newcommand{\R}{\mathbb R}
\newcommand{\E}{\mathbb E}
\newcommand{\Prob}{\mathbb P}
\newcommand{\haf}{\operatorname{haf}}
\newcommand{\CN}{\mathcal{CN}}
\newcommand{\Lt}{\mathrm{Lt}}
\newcommand{\norm}[1]{\left\lVert #1\right\rVert}
\newcommand{\abs}[1]{\left\lvert #1\right\rvert}

\newcommand{\dd}{\,\mathrm d}

\title{Anticoncentration of Complex Gaussian Hafnians}
\author{Priyanshu Pant\\
Indian Institute of Technology Indore, India\\
\texttt{priyanshupant03@gmail.com}}
\date{September 2026}

\begin{document}
\maketitle

\begin{abstract}
Let \(G_{2n}\) be a complex symmetric random matrix whose entries above the diagonal are independent standard circular complex Gaussians, and let \(H_n=\operatorname{haf}(G_{2n})\). We prove the uniform shifted anticoncentration bound
\[
\Pr\!\left(
\left|
\frac{H_n}{\sqrt{(2n-1)!!}}-z
\right|
\le \varepsilon
\right)
\le
2\sqrt{\frac n\pi}\,\varepsilon^2
\]
for every \(z\in\mathbb C\) and \(\varepsilon>0\). This establishes a local anticoncentration property that supports hardness arguments for quantum advantage in Gaussian boson sampling.
\end{abstract}

\section{Introduction}\label{sec:intro}

Boson sampling and Gaussian boson sampling are prominent proposals for
exhibiting quantum computational advantage with restricted photonic devices.
In the original boson-sampling model, output probabilities are governed by
matrix permanents.  Aaronson and Arkhipov showed that the passage from exact
to approximate sampling requires more than the worst-case hardness of the
permanent: their reduction isolates, among other ingredients, an
anticoncentration property for permanents of Gaussian random matrices
\cite{AaronsonArkhipov2013}.  Informally, one needs typical random output
weights not to become so small that an additive approximation contains no
useful relative information.  This led to the Permanent Anti-Concentration
Conjecture and, more generally, to the study of the distribution of random
matrix polynomials at their natural second-moment scale.

Gaussian boson sampling (GBS) replaces single-photon inputs by Gaussian
squeezed states.  Hamilton et al. showed that collision-free GBS output
probabilities are expressed through squared hafnians of complex symmetric
matrices \cite{HamiltonEtAl2017}; see also the detailed development in
\cite{KruseEtAl2019}.  The hafnian is the perfect-matching polynomial for a
symmetric matrix.  If $A=(a_{ij})$ is symmetric of order $2n$, then
\begin{equation}\label{eq:haf-def}
 \haf(A)=\sum_{M\in\mathcal M_{2n}}
             \prod_{\{i,j\}\in M}a_{ij},
\end{equation}
where $\mathcal M_{2n}$ denotes the set of perfect matchings of
$\{1,\ldots,2n\}$.  Thus the hafnian plays for general perfect matchings the
role played by the permanent for bipartite perfect matchings.

The complexity-theoretic evidence for approximate GBS has subsequently been
organized around the same broad ingredients as in ordinary boson sampling:
hiding random instances inside the physical ensemble, average-case hardness
of the relevant output probabilities, and statistical control of the output
distribution; see, for example, \cite{DeshpandeEtAl2022}.  Recent published
work of Ehrenberg et al. studies the latter question through moments of GBS
output probabilities and reveals a transition between lack of
anticoncentration and weak anticoncentration as the number of squeezed modes
varies with the observed photon number \cite{EhrenbergEtAl2025PRL,EhrenbergEtAl2025PRA}.
These results are moment based.  They do not by themselves give local control
of how much probability a random hafnian can place in a small neighborhood of
a prescribed complex value.

The question studied here is this local, shifted form of anticoncentration.
Let $G_{2n}$ be a complex symmetric random matrix whose entries above the
diagonal are independent standard circular complex Gaussians
$\CN(0,1)$.  The diagonal is irrelevant for the hafnian and may be chosen
arbitrarily.  Set
\begin{equation}\label{eq:basic-defs}
 H_n=\haf(G_{2n}),\qquad
 M_n=(2n-1)!!,\qquad
 W_n=\frac{H_n}{\sqrt{M_n}}.
\end{equation}
The normalization is the root-mean-square scale, since
$\E|H_n|^2=M_n$; see \cref{lem:second-moment}.  We ask for a bound, uniform in
$z\in\C$, on
\[
       \Prob\bigl( |W_n-z|\leq\varepsilon\bigr).
\]
Because the ambient space is two-dimensional, the natural local scale is
$\varepsilon^2$.

Our first main result gives such a bound with a polynomial loss in $n$.

\begin{theorem}[Complex Gaussian hafnian anticoncentration]\label{thm:main}
For every $n\geq1$, the random variable $W_n$ has a radial, radially
nonincreasing density $p_n$ on $\C$.  Moreover,
\begin{equation}\label{eq:density-main}
 \norm{p_n}_{\infty}=p_n(0)
 \leq \frac{1}{\pi}B_n,
 \qquad
 B_n:=\frac{(2n-1)!!}{2^{n-1}(n-1)!}
      =\frac{2n}{4^n}\binom{2n}{n}.
\end{equation}
Consequently, for every $z\in\C$ and every $\varepsilon>0$,
\begin{equation}\label{eq:main-small-ball}
 \Prob\bigl(\abs{W_n-z}\leq\varepsilon\bigr)
 \leq B_n\varepsilon^2
 \leq 2\sqrt{\frac n\pi}\,\varepsilon^2.
\end{equation}
\end{theorem}

The proof begins with a one-vertex expansion.  If $B_{2n-1}$ is an odd
symmetric Gaussian matrix and
\[
  (C_n)_j=\haf((B_{2n-1})_{-j}),\qquad 1\le j\le 2n-1,
\]
is its hafnian cofactor vector, then adding one new Gaussian vertex gives
\begin{equation}\label{eq:intro-one-vertex}
           H_n=g^{\mathsf T}C_n,
           \qquad g\sim\CN(0,I_{2n-1}),
\end{equation}
with $g$ independent of $C_n$.  Hence, conditional on $C_n$,
\[
             H_n\mid C_n\sim\CN(0,\norm{C_n}^2).
\]
The density at the origin is therefore
\begin{equation}\label{eq:intro-inverse}
        p_{H_n}(0)=\frac1\pi\E\frac1{\norm{C_n}^2}.
\end{equation}
Thus the small-ball problem becomes an inverse-norm problem for an odd
hafnian cofactor vector.

The structural step is a two-vertex expansion.  After exposing two vertices,
a linear combination of the cofactors has the form
\begin{equation}\label{eq:intro-two-vertex}
  x^{\mathsf T}Ty+w_2x^{\mathsf T}\ell+w_1y^{\mathsf T}\ell+aq.
\end{equation}
Here $x$ and $y$ are independent Gaussian vectors, while $a$ is the
independent Gaussian weight of the edge joining the two exposed vertices.
The final term is precisely the new feature caused by symmetry.  Averaging
$a$ turns it into the nonnegative factor
\begin{equation}\label{eq:intro-positive}
        \E_a e^{i\operatorname{Re}(aq)}=e^{-|q|^2/4}.
\end{equation}
The remaining Gaussian integral admits an exact two-coordinate geometric
interpolation.  Since the shared-edge factor is nonnegative, H\"older's
inequality preserves the interpolation after averaging over the unexposed
matrix.  Iterating the resulting compression shows that, at fixed Fourier
energy, one coordinate is extremal.  Applied to the inverse moment
\[
       A_n:=\E\norm{C_n}^{-2},
\]
this gives the simple recurrence
\begin{equation}\label{eq:intro-recurrence}
        A_{n+1}\leq\frac{A_n}{2n}.
\end{equation}
Since $A_1=1$, one obtains
$A_n\leq [2^{n-1}(n-1)!]^{-1}$, from which \cref{thm:main} follows.

The compression theorem contains more information than the single inverse
moment used above.  Our second main result gives an explicit stochastic
comparison.  For nonnegative random variables $U,V$, write
$U\leq_{\Lt}V$ if
\[
          \E e^{-sU}\geq\E e^{-sV}\qquad(s\geq0).
\]

\begin{theorem}[Product-Gamma Laplace domination]\label{thm:intro-lt}
Let $\gamma_r\sim\operatorname{Gamma}(r,1)$ be independent random variables.
Then, for every $n\geq1$,
\begin{equation}\label{eq:intro-cofactor-lt}
  \prod_{j=2}^{n}\gamma_{2j-1}
  \leq_{\Lt}\norm{C_n}^2,
\end{equation}
where the empty product for $n=1$ equals $1$.  Consequently,
\begin{equation}\label{eq:intro-haf-lt}
  \prod_{j=1}^{n}\gamma_{2j-1}
  \leq_{\Lt}|H_n|^2.
\end{equation}
\end{theorem}

The odd Gamma shapes $1,3,\ldots,2n-1$ record the dimensions of the
successive Gaussian cofactor vectors.  Besides recovering the inverse-moment
estimate needed for \cref{thm:main}, \cref{thm:intro-lt} yields an entire
family of fractional negative-moment bounds; see \cref{cor:fractional-moments}.

Our proof is close in spirit to the recent Gaussian permanent
anticoncentration argument of Koehler and Leung \cite{KoehlerLeung2026}: both
reduce a high-dimensional cofactor problem by Gaussian interpolation.  The
symmetric hafnian setting has the additional shared-edge term in
\eqref{eq:intro-two-vertex}; the positivity in \eqref{eq:intro-positive} is
what allows the compression to survive this extra coupling.  Concurrent work
of Zhao also considers shifted anticoncentration for Gaussian hafnians,
including the complex symmetric Gaussian ensemble \cite{ZhaoUniform2026}.

Finally, the independent symmetric Gaussian ensemble is directly relevant to
the maximally squeezed GBS model.  A theorem of Shou, Miller, and Galitski
shows that sufficiently small blocks of a circular orthogonal ensemble (COE),
after rescaling, are close to an independent symmetric complex Gaussian
matrix, both in total variation and, in a smaller range, by a pointwise
density comparison \cite{ShouMillerGalitski2025}.  Combining this hiding
result with \cref{thm:main} transfers the local $\varepsilon^2$ small-ball
behavior to COE hafnians.  This is a probabilistic ingredient in the GBS
hardness program; it does not by itself establish classical sampling
hardness, for which average-case hardness remains a separate requirement.

The rest of the paper is organized as follows.  In \cref{sec:reduction} we
reduce the density problem to an inverse moment of the cofactor vector.  In
\cref{sec:compression} we prove the two-vertex decomposition and the
coordinate-compression theorem.  The direct inverse-moment recurrence and
\cref{thm:main} are proved in \cref{sec:anticoncentration}.  The stronger
product-Gamma comparison is developed in \cref{sec:laplace}.  The GBS
application is given in \cref{sec:gbs}.

\section{From a Gaussian hafnian to one inverse moment}\label{sec:reduction}

We begin with the elementary expansion identity that drives the proof.

\begin{lemma}[One-vertex hafnian expansion]\label{lem:haf-expansion}
Let $A=(a_{ij})$ be a symmetric $2n\times 2n$ matrix.  Then
\begin{equation}\label{eq:haf-expansion}
  \haf(A)=\sum_{j=2}^{2n}a_{1j}\,
           \haf(A_{-\{1,j\}}).
\end{equation}
\end{lemma}

\begin{proof}
Every perfect matching pairs vertex $1$ with a unique vertex
$j\in\{2,\ldots,2n\}$.  Once the edge $\{1,j\}$ is fixed, the remaining
vertices contribute a perfect matching of the principal submatrix obtained by
deleting $1$ and $j$.  Summing over the possible partners $j$ gives
\eqref{eq:haf-expansion}.
\end{proof}

Let $B_{2n-1}$ be a symmetric $(2n-1)\times(2n-1)$ matrix whose entries above
the diagonal are independent $\CN(0,1)$ variables.  Define its hafnian
cofactor vector
\begin{equation}\label{eq:cofactor-def}
   C_n=\bigl((C_n)_1,\ldots,(C_n)_{2n-1}\bigr)^{\mathsf T},
   \qquad
   (C_n)_j:=\haf((B_{2n-1})_{-j}).
\end{equation}
We use the convention $\haf(\varnothing)=1$, so $C_1=(1)$.

\begin{lemma}[Conditional Gaussian representation]\label{lem:conditional-gaussian}
Let $g\sim\CN(0,I_{2n-1})$ be independent of $B_{2n-1}$.  If $G_{2n}$ is
obtained by adjoining to $B_{2n-1}$ a new vertex whose incident edge vector is
$g$, then
\begin{equation}\label{eq:H-C}
        H_n=g^{\mathsf T}C_n.
\end{equation}
Consequently, conditional on $C_n$,
\begin{equation}\label{eq:conditional-law}
        H_n\mid C_n\sim\CN(0,\norm{C_n}^2).
\end{equation}
\end{lemma}

\begin{proof}
The identity \eqref{eq:H-C} is \cref{lem:haf-expansion} applied at the new
vertex.  Conditional on $C_n$, the right-hand side is a complex linear
combination of independent standard circular Gaussians.  Its conditional
variance is
\[
  \E\bigl[|g^{\mathsf T}C_n|^2\mid C_n\bigr]
   =\sum_{j=1}^{2n-1}|(C_n)_j|^2
   =\norm{C_n}^2,
\]
which gives \eqref{eq:conditional-law}.
\end{proof}

The cofactor norm is nonzero almost surely.

\begin{lemma}[Almost sure positivity]\label{lem:positive-cofactor}
For every $n\geq1$,
\[
             \Prob(\norm{C_n}>0)=1.
\]
\end{lemma}

\begin{proof}
For $n=1$ this is immediate from $C_1=(1)$.  For $n\ge2$, each coordinate
$(C_n)_j$ is a nonzero polynomial in the independent Gaussian entries of
$B_{2n-1}$: any fixed perfect matching of the remaining $2n-2$ vertices
contributes a monomial with coefficient one.  A nonzero polynomial in
continuously distributed real variables vanishes on a set of Lebesgue measure
zero.  Hence, for instance, $(C_n)_1\ne0$ almost surely.
\end{proof}

\begin{proposition}[Gaussian-mixture density]\label{prop:mixture-density}
Let
\[
              V_n:=\norm{C_n}^2.
\]
Then $H_n$ has density
\begin{equation}\label{eq:mixture-density}
 p_{H_n}(z)
 =\E\left[
      \frac1{\pi V_n}\exp\!\left(-\frac{|z|^2}{V_n}\right)
        \right],\qquad z\in\C.
\end{equation}
In particular, $p_{H_n}$ is radial and nonincreasing in $|z|$, and
\begin{equation}\label{eq:density-origin}
   \norm{p_{H_n}}_\infty=p_{H_n}(0)
   =\frac1\pi\E V_n^{-1}.
\end{equation}
\end{proposition}

\begin{proof}
By \cref{lem:conditional-gaussian,lem:positive-cofactor}, conditional on
$C_n$ the random variable $H_n$ has the density
\[
   \frac1{\pi V_n}\exp\!\left(-\frac{|z|^2}{V_n}\right).
\]
Averaging this conditional density gives \eqref{eq:mixture-density}.  Every
term in the mixture is radial and nonincreasing in $|z|$, so the same is true
of the mixture.  Setting $z=0$ gives \eqref{eq:density-origin}.
\end{proof}

We therefore define the single quantity that will control the entire
anticoncentration problem:
\begin{equation}\label{eq:An}
             A_n:=\E\frac1{\norm{C_n}^2}.
\end{equation}
The next section proves the compression theorem from which the recurrence for
$A_n$ follows.

\section{Two-vertex compression of hafnian cofactors}\label{sec:compression}

For $z,w\in\C^d$, write
\[
        (z,w)_\R:=\operatorname{Re}\sum_{j=1}^d\overline z_jw_j.
\]
The characteristic function of the cofactor vector is
\begin{equation}\label{eq:Phi-def}
        \Phi_n(z):=\E e^{i(z,C_n)_\R},
        \qquad z\in\C^{2n-1}.
\end{equation}
We will prove that at fixed Euclidean norm this characteristic function is
largest in magnitude along a coordinate axis.

\subsection{The two-vertex algebra}

Let $m=2n-1$ and expose vertices $1$ and $2$ of $B_m$.  With
$R=\{3,\ldots,m\}$, write
\begin{equation}\label{eq:block-B}
 B_m=
 \begin{pmatrix}
  *&a&x^{\mathsf T}\\
  a&*&y^{\mathsf T}\\
  x&y&D
 \end{pmatrix}.
\end{equation}
Here $a\sim\CN(0,1)$, $x$ and $y$ are independent standard complex Gaussian
vectors indexed by $R$, and $D$ is an independent symmetric Gaussian matrix.
For $r\in R$ put
\begin{equation}\label{eq:ell-def}
                 \ell_r:=\haf(D_{-r}).
\end{equation}

\begin{lemma}[Two-vertex cofactor decomposition]\label{lem:two-vertex}
With the notation above,
\begin{equation}\label{eq:first-two}
   (C_n)_1=y^{\mathsf T}\ell,
   \qquad
   (C_n)_2=x^{\mathsf T}\ell.
\end{equation}
For every $j\in R$,
\begin{equation}\label{eq:remaining-cofactor}
 (C_n)_j
 =a\ell_j+
   \sum_{\substack{r,s\in R\setminus\{j\}\\ r\ne s}}
      x_r y_s\,\haf(D_{-\{j,r,s\}}).
\end{equation}
Consequently, for arbitrary coefficients $w=(w_1,\ldots,w_m)\in\C^m$,
there is a matrix $T=T(D,w_R)$ such that
\begin{equation}\label{eq:bilinear-decomp}
  \sum_{j=1}^{m}w_j(C_n)_j
  =x^{\mathsf T}Ty+w_2x^{\mathsf T}\ell
   +w_1y^{\mathsf T}\ell+aq,
  \qquad
  q:=\sum_{j\in R}w_j\ell_j.
\end{equation}
\end{lemma}

\begin{proof}
Deleting vertex $1$ leaves vertex $2$ together with $R$.  Expanding the
resulting hafnian at vertex $2$ gives
$(C_n)_1=\sum_{r\in R}y_r\ell_r$.  The second identity follows symmetrically.

Fix $j\in R$.  In a perfect matching contributing to
$\haf((B_m)_{-j})$, either vertices $1$ and $2$ are matched together, which
contributes $a\ell_j$, or vertex $1$ is matched to some
$r\in R\setminus\{j\}$ and vertex $2$ to a distinct
$s\in R\setminus\{j\}$.  In the latter case the remaining contribution is
$x_ry_s\haf(D_{-\{j,r,s\}})$.  This proves
\eqref{eq:remaining-cofactor}.  After multiplication by $w_j$ and summation
over $j$, all terms quadratic in $x$ and $y$ can be collected into a bilinear
form $x^{\mathsf T}Ty$, while the terms involving $a$ sum to $aq$.
\end{proof}

The last term in \eqref{eq:bilinear-decomp} is the only part not present in
the analogous rectangular permanent calculation.  Its Gaussian average is
nonnegative.

\begin{lemma}[Positive shared-edge averaging]\label{lem:shared-edge}
Conditioned on $D$, and hence on $q$,
\begin{equation}\label{eq:shared-edge}
        \E_a e^{i\operatorname{Re}(aq)}
        =e^{-|q|^2/4}\ge0.
\end{equation}
\end{lemma}

\begin{proof}
Write $a=(X+iY)/\sqrt2$ with $X,Y$ independent $N(0,1)$.  Then
$\operatorname{Re}(aq)$ is a centered real Gaussian of variance $|q|^2/2$.
The standard characteristic-function formula gives
$\E e^{itN}=e^{-t^2\sigma^2/2}$ for $N\sim N(0,\sigma^2)$, yielding
\eqref{eq:shared-edge}.
\end{proof}

\subsection{A Gaussian interpolation lemma}

The next elementary identity is the analytic input to the compression.  It is
the same Gaussian geometric interpolation used in the permanent setting
\cite{KoehlerLeung2026}.

\begin{lemma}[Gaussian geometric interpolation]\label{lem:interpolation}
Let $X,Y$ be independent standard Gaussian vectors in $\R^d$.  Let
$T\in\R^{d\times d}$, $L\in\R^{d\times k}$, and $u,v\in\R^k$.  Define
\[
 F_{T,L}(u,v)
 =\E_{X,Y}\exp i\bigl(
        X^{\mathsf T}TY+X^{\mathsf T}Lv+Y^{\mathsf T}Lu
                     \bigr).
\]
Then $F_{T,L}(u,0)$ and $F_{T,L}(0,v)$ are positive real numbers.  Moreover,
for $0\le\theta\le1$,
\begin{equation}\label{eq:gaussian-interpolation}
 \left|F_{T,L}(\sqrt\theta\,u,
                 \sqrt{1-\theta}\,v)\right|
 =F_{T,L}(u,0)^\theta F_{T,L}(0,v)^{1-\theta}.
\end{equation}
\end{lemma}

\begin{proof}
Put $p=Lv$, $q=Lu$, and $A=I+T^{\mathsf T}T$.  Averaging first over $X$,
\[
 F_{T,L}(u,v)
 =e^{-\norm p^2/2}
  \E_Y\exp\left(
   -\frac12Y^{\mathsf T}T^{\mathsf T}TY
   -p^{\mathsf T}TY+iq^{\mathsf T}Y
             \right).
\]
Using the standard Gaussian integral
\[
 \E_Z\exp\left(-\frac12 Z^{\mathsf T}QZ+h^{\mathsf T}Z\right)
 =\det(I+Q)^{-1/2}
   \exp\left(\frac12 h^{\mathsf T}(I+Q)^{-1}h\right)
\]
with $h=-T^{\mathsf T}p+iq$, and taking absolute values, gives
\begin{align}\label{eq:F-modulus}
 |F_{T,L}(u,v)|
  ={}&\det(A)^{-1/2}\nonumber\\
 &\times\exp\left(
   -\frac12p^{\mathsf T}(I+TT^{\mathsf T})^{-1}p
   -\frac12q^{\mathsf T}A^{-1}q
              \right).
\end{align}
Here we used
$I-T(I+T^{\mathsf T}T)^{-1}T^{\mathsf T}=(I+TT^{\mathsf T})^{-1}$.
Setting $u=0$ or $v=0$ shows that the endpoint values are positive.  Since
the two quadratic forms in \eqref{eq:F-modulus} scale respectively by
$1-\theta$ and $\theta$, \eqref{eq:gaussian-interpolation} follows.
\end{proof}

\begin{corollary}[Weighted interpolation]\label{cor:weighted}
Let $(T,L,W)$ be random with $W\ge0$, and assume the displayed expectations
are finite.  Then, for $0\le\theta\le1$,
\begin{align}\label{eq:weighted-interpolation}
 &\E\!\left[
 W\left|F_{T,L}(\sqrt\theta\,u,
                  \sqrt{1-\theta}\,v)\right|
           \right]\nonumber\\
 &\qquad\le
 \bigl(\E[W F_{T,L}(u,0)]\bigr)^\theta
 \bigl(\E[W F_{T,L}(0,v)]\bigr)^{1-\theta}.
\end{align}
\end{corollary}

\begin{proof}
By \cref{lem:interpolation}, the integrand equals
\[
 (WF_{T,L}(u,0))^\theta
 (WF_{T,L}(0,v))^{1-\theta}.
\]
Apply H\"older's inequality.
\end{proof}

\begin{remark}\label{rem:complex-realification}
The same statement applies to standard complex Gaussian vectors whenever the
phase is the real part of a complex bilinear expression.  Indeed, identify
$\C^d$ with $\R^{2d}$, write real and imaginary parts, and rescale by
$\sqrt2$ so that the resulting real Gaussian vectors are standard.  Complex
multiplication is represented by a real linear map, so all numerical factors
are absorbed into the matrices $T$ and $L$.
\end{remark}

\subsection{Coordinate compression}

\begin{proposition}[Hafnian cofactor compression]\label{prop:compression}
For every $n\ge2$ and $z\in\C^{2n-1}$,
\begin{equation}\label{eq:compression}
             |\Phi_n(z)|\leq \Phi_n(\norm z e_1).
\end{equation}
The quantity on the right is real and nonnegative.
\end{proposition}

\begin{proof}
It is convenient to write $w_j=\overline z_j$, so that
$(z,C_n)_\R=\operatorname{Re}\sum_jw_j(C_n)_j$.  It is enough to show that
any two nonzero coordinates of $w$ can be compressed into one coordinate
without decreasing the upper bound for $|\Phi_n(z)|$ and while preserving
$\norm z$.

Relabel the two coordinates as $1$ and $2$ and use the decomposition
\eqref{eq:bilinear-decomp}.  Condition on $D$.  By
\cref{lem:shared-edge}, averaging the scalar edge $a$ contributes the
nonnegative weight
\[
              W(D)=e^{-|q|^2/4}.
\]
After identifying the complex Gaussian vectors with real Gaussian
vectors as in \cref{rem:complex-realification}, the conditional expectation
over $x,y$ is of the form $F_{T_D,L_D}(w_1,w_2)$.  The same real map $L_D$
occurs in both linear terms because both use the vector $\ell$.

Put
\[
 r=(|w_1|^2+|w_2|^2)^{1/2},
 \qquad
 \theta=\frac{|w_1|^2}{r^2}.
\]
When both coordinates are nonzero, define
\[
   u=r\frac{w_1}{|w_1|},
   \qquad
   v=r\frac{w_2}{|w_2|}.
\]
Thus $w_1=\sqrt\theta\,u$ and
$w_2=\sqrt{1-\theta}\,v$ in the realified variables.  By the triangle
inequality followed by \cref{cor:weighted},
\begin{align*}
 |\Phi_n(z)|
 &\le
 \E_D\left[W(D)|F_{T_D,L_D}(w_1,w_2)|\right]\\
 &\le
 \bigl(\E_D[W(D)F_{T_D,L_D}(u,0)]\bigr)^\theta
 \bigl(\E_D[W(D)F_{T_D,L_D}(0,v)]\bigr)^{1-\theta}.
\end{align*}
The two expectations on the last line are precisely the characteristic
functions obtained by replacing $(w_1,w_2)$ by $(u,0)$ and $(0,v)$,
respectively.  Both are nonnegative: conditionally on $D$ this follows from
$W(D)\ge0$ and the endpoint positivity in \cref{lem:interpolation}.  Their
weighted geometric mean is therefore at most their maximum.  Hence one of
the two replacements does not decrease the upper bound, preserves
$|w_1|^2+|w_2|^2$, and reduces the number of nonzero coordinates by one.

Iterating leaves a vector with a single nonzero coordinate of modulus
$\norm z$.  The coordinates of $C_n$ are exchangeable, so this coordinate
may be taken to be the first.  Moreover, $(C_n)_1$ has the same law as
$H_{n-1}$, which is circularly symmetric: multiplying all Gaussian edges
incident to a fixed vertex by a unit complex scalar preserves the matrix law
and multiplies every hafnian monomial by that scalar.  Hence the phase of the
remaining coordinate may be removed.  By \cref{lem:conditional-gaussian},
for $r\ge0$,
\[
 \Phi_n(re_1)
 =\E e^{i\operatorname{Re}(rH_{n-1})}
 =\E\exp\left(-\frac{r^2}{4}\norm{C_{n-1}}^2\right)\ge0.
\]
This proves \eqref{eq:compression}.
\end{proof}

\begin{remark}[Where symmetry enters]\label{rem:symmetry}
The factor $e^{-|q|^2/4}$ in \eqref{eq:shared-edge} is the only additional
piece created by the symmetric hafnian geometry.  Before averaging $a$, the
term $aq$ couples the two exposed coordinates to all remaining coefficients.
Gaussian averaging removes this phase and leaves a positive weight; H\"older's
inequality then preserves the one-sided compression.
\end{remark}

\section{The inverse-moment recurrence and anticoncentration}\label{sec:anticoncentration}

We now use \cref{prop:compression} to prove the main theorem directly,
without first introducing Laplace-transform order.

A standard complex Gaussian vector $g\sim\CN(0,I_d)$ satisfies
\begin{equation}\label{eq:gaussian-char-vector}
  \E_g e^{i(z,g)_\R}=e^{-\norm z^2/4},
  \qquad
  \norm g^2\sim\operatorname{Gamma}(d,1).
\end{equation}
Consequently, for every fixed $c\in\C^d$ and $s\ge0$,
\begin{equation}\label{eq:linearization}
 e^{-s\norm c^2}
 =\E_g e^{i(2\sqrt{s}\,g,c)_\R}.
\end{equation}
We also use
\begin{equation}\label{eq:gamma-inverse}
 \E\frac1{\norm g^2}=\frac1{d-1},\qquad d\ge2,
\end{equation}
which follows by integrating the $\operatorname{Gamma}(d,1)$ density.

\begin{proposition}[Inverse-moment recurrence]\label{prop:recurrence}
For every $n\ge1$,
\begin{equation}\label{eq:recurrence}
        A_{n+1}\leq\frac{A_n}{2n}.
\end{equation}
In particular, all the inverse moments $A_n$ are finite.
\end{proposition}

\begin{proof}
Let $g\sim\CN(0,I_{2n+1})$ be independent of $C_{n+1}$.  By
\eqref{eq:linearization},
\[
 \E e^{-s\norm{C_{n+1}}^2}
 =\E_g\Phi_{n+1}(2\sqrt{s}\,g).
\]
The left-hand side is real and nonnegative.  Hence, using
\cref{prop:compression},
\begin{align}
 \E e^{-s\norm{C_{n+1}}^2}
 &\leq \E_g\left|\Phi_{n+1}(2\sqrt{s}\,g)\right|\nonumber\\
 &\leq \E_g\Phi_{n+1}(2\sqrt{s}\,\norm g\,e_1).
 \label{eq:recurrence-laplace-1}
\end{align}
The first coordinate of $C_{n+1}$ has the same law as $H_n$.  Therefore, by
\cref{lem:conditional-gaussian}, for $r\ge0$,
\[
 \Phi_{n+1}(re_1)
 =\E\exp\left(-\frac{r^2}{4}\norm{C_n}^2\right).
\]
Substituting $r=2\sqrt{s}\,\norm g$ into
\eqref{eq:recurrence-laplace-1} gives
\begin{equation}\label{eq:recurrence-laplace}
 \E e^{-s\norm{C_{n+1}}^2}
 \leq
 \E e^{-s\norm g^2\norm{C_n}^2}.
\end{equation}
Now use $x^{-1}=\int_0^\infty e^{-sx}\,\dd s$ and Tonelli's theorem.  Since
all integrands are nonnegative,
\begin{align*}
 A_{n+1}
 &=\int_0^\infty \E e^{-s\norm{C_{n+1}}^2}\,\dd s\\
 &\leq
 \E\frac1{\norm g^2\norm{C_n}^2}
 =\E\norm g^{-2}\,A_n.
\end{align*}
The two factors are independent.  Since $g\in\C^{2n+1}$,
\eqref{eq:gamma-inverse} gives $\E\norm g^{-2}=1/(2n)$, proving
\eqref{eq:recurrence}.  The base case $A_1=1$ then also proves finiteness
inductively.
\end{proof}

\begin{corollary}[Inverse-moment bound]\label{cor:inverse-bound}
For every $n\ge1$,
\begin{equation}\label{eq:inverse-bound}
       \E\frac1{\norm{C_n}^2}
       \leq \frac1{2^{n-1}(n-1)!}.
\end{equation}
\end{corollary}

\begin{proof}
Since $C_1=(1)$, $A_1=1$.  Iterating \cref{prop:recurrence} gives
\[
 A_n\leq\frac1{2\cdot4\cdots(2n-2)}
      =\frac1{2^{n-1}(n-1)!}.
\]
\end{proof}

We next record the normalization.

\begin{lemma}[Exact second moment]\label{lem:second-moment}
For every $n\ge1$,
\begin{equation}\label{eq:second-moment}
                \E|H_n|^2=(2n-1)!!.
\end{equation}
\end{lemma}

\begin{proof}
Expand
\[
 H_n=\sum_{M\in\mathcal M_{2n}}X_M,
 \qquad
 X_M:=\prod_{\{i,j\}\in M}(G_{2n})_{ij}.
\]
If $M\ne M'$, some centered Gaussian edge variable occurs in one monomial but
not the other, so independence and centering imply
$\E[X_M\overline{X_{M'}}]=0$.  For $M=M'$, every edge has second moment one,
so $\E|X_M|^2=1$.  There are $(2n-1)!!$ perfect matchings.
\end{proof}

\begin{proof}[Proof of \cref{thm:main}]
By \cref{prop:mixture-density,cor:inverse-bound},
\begin{equation}\label{eq:unnormalized-density}
 \norm{p_{H_n}}_\infty=p_{H_n}(0)
 \leq\frac1{\pi\,2^{n-1}(n-1)!}.
\end{equation}
By \cref{lem:second-moment}, $W_n=H_n/\sqrt{M_n}$ with
$M_n=(2n-1)!!$.  Scaling the complex plane by $1/\sqrt{M_n}$ multiplies a
two-dimensional density by $M_n$, hence
\[
 \norm{p_n}_\infty
 \leq\frac1\pi\frac{(2n-1)!!}{2^{n-1}(n-1)!}
 =\frac{B_n}{\pi}.
\]
Radial monotonicity is preserved by this scaling.  Integrating the density
over a disk of radius $\varepsilon$ centered at any $z\in\C$ gives
\[
 \Prob(|W_n-z|\leq\varepsilon)
 \leq\pi\varepsilon^2\norm{p_n}_\infty
 \leq B_n\varepsilon^2.
\]
Finally,
\[
 B_n=\frac{2n}{4^n}\binom{2n}{n}
 \leq 2\sqrt{\frac n\pi},
\]
by the standard Wallis bound
$\binom{2n}{n}\leq4^n/\sqrt{\pi n}$.
\end{proof}

\section{A stronger product-Gamma comparison}\label{sec:laplace}

The preceding proof used compression only through one inverse moment.  We now
show that the same mechanism yields a full Laplace-transform comparison.

For nonnegative random variables $U,V$, write
\begin{equation}\label{eq:lt-order-def}
 U\leq_{\Lt}V
 \quad\Longleftrightarrow\quad
 \E e^{-sU}\geq\E e^{-sV}
 \quad\text{for every }s\ge0.
\end{equation}
The following Fourier-to-Laplace implication is standard and is included for
completeness.

\begin{lemma}[Fourier-to-Laplace comparison]\label{lem:fourier-laplace}
Let $U,V$ be random vectors in $\C^d$.  Suppose
\begin{equation}\label{eq:fourier-domination}
  \operatorname{Re}\E e^{i(z,U)_\R}
  \leq
  \operatorname{Re}\E e^{i(z,V)_\R}
  \qquad\text{for all }z\in\C^d.
\end{equation}
Then
\begin{equation}\label{eq:norm-lt}
                \norm V^2\leq_{\Lt}\norm U^2.
\end{equation}
\end{lemma}

\begin{proof}
Identify $\C^d$ with $\R^{2d}$.  For $s>0$, the Fourier transform of
$f_s(x)=e^{-s\norm x^2}$ is a positive Gaussian:
\[
 \widehat f_s(z)=\left(\frac\pi s\right)^d
                 e^{-\norm z^2/(4s)}.
\]
Fourier inversion and Fubini therefore give
\[
 \E e^{-s\norm U^2}
 =c_{d,s}\int_{\R^{2d}}e^{-\norm z^2/(4s)}
   \operatorname{Re}\E e^{i(z,U)_\R}\,\dd z,
\]
where $c_{d,s}>0$.  The same identity holds for $V$.  Since the Fourier
weight is nonnegative, \eqref{eq:fourier-domination} implies
\[
       \E e^{-s\norm U^2}\leq\E e^{-s\norm V^2},
\]
which is \eqref{eq:norm-lt}.
\end{proof}

Let $\gamma_1,\gamma_3,\gamma_5,\ldots$ be independent with
$\gamma_r\sim\operatorname{Gamma}(r,1)$, and define
\begin{equation}\label{eq:Pi-def}
      \Pi_1=1,
      \qquad
      \Pi_n=\prod_{j=2}^n\gamma_{2j-1}\quad(n\ge2).
\end{equation}

\begin{theorem}[Cofactor Laplace domination]\label{thm:cofactor-lt}
For every $n\ge1$,
\begin{equation}\label{eq:cofactor-lt}
                \Pi_n\leq_{\Lt}\norm{C_n}^2.
\end{equation}
Consequently,
\begin{equation}\label{eq:hafnian-lt}
       \prod_{j=1}^n\gamma_{2j-1}
       \leq_{\Lt}|H_n|^2.
\end{equation}
\end{theorem}

\begin{proof}
We prove \eqref{eq:cofactor-lt} by induction.  For $n=1$,
$C_1=(1)$ and $\Pi_1=1$, so equality holds.

Assume \eqref{eq:cofactor-lt} for $n$.  The first coordinate of $C_{n+1}$
has the same distribution as $H_n$.  By \cref{lem:conditional-gaussian},
for $r\ge0$,
\begin{align}\label{eq:endpoint-char-lt}
 \Phi_{n+1}(re_1)
 &=\E\exp\left(-\frac{r^2}{4}\norm{C_n}^2\right)\nonumber\\
 &\leq\E\exp\left(-\frac{r^2}{4}\Pi_n\right),
\end{align}
where the inequality is the induction hypothesis.  By
\cref{prop:compression}, for $z\in\C^{2n+1}$,
\begin{align*}
 \operatorname{Re}\Phi_{n+1}(z)
 &\leq|\Phi_{n+1}(z)|\\
 &\leq\Phi_{n+1}(\norm z e_1)\\
 &\leq\E\exp\left(-\frac{\norm z^2}{4}\Pi_n\right).
\end{align*}
If $g_{2n+1}\sim\CN(0,I_{2n+1})$ is independent of $\Pi_n$, the last
quantity is the characteristic function of $\sqrt{\Pi_n}\,g_{2n+1}$ at
$z$.  Hence \cref{lem:fourier-laplace} gives
\[
 \Pi_n\norm{g_{2n+1}}^2
 \leq_{\Lt}\norm{C_{n+1}}^2.
\]
Since $\norm{g_{2n+1}}^2\sim\operatorname{Gamma}(2n+1,1)$ independently of
$\Pi_n$, the left-hand side has the law of $\Pi_{n+1}$.

Finally, \cref{lem:conditional-gaussian} implies the distributional identity
\begin{equation}\label{eq:H-square-mixture}
       |H_n|^2\stackrel d=\gamma_1\norm{C_n}^2,
\end{equation}
where $\gamma_1\sim\operatorname{Gamma}(1,1)$ is independent.  Laplace order
is preserved by multiplication by a common independent nonnegative random
variable: condition on that multiplier.  Multiplying
\eqref{eq:cofactor-lt} by an independent $\gamma_1$ proves
\eqref{eq:hafnian-lt}.
\end{proof}

The main inverse-moment estimate is only one consequence of this order.

\begin{corollary}[Fractional negative moments]\label{cor:fractional-moments}
For $n\ge2$ and $0<q<3$,
\begin{equation}\label{eq:cofactor-negative-q}
 \E\norm{C_n}^{-2q}
 \leq
 \prod_{j=2}^n
   \frac{\Gamma(2j-1-q)}{\Gamma(2j-1)}.
\end{equation}
For every $n\ge1$ and $0<q<1$,
\begin{equation}\label{eq:haf-negative-q}
 \E|H_n|^{-2q}
 \leq
 \prod_{j=1}^n
   \frac{\Gamma(2j-1-q)}{\Gamma(2j-1)}.
\end{equation}
\end{corollary}

\begin{proof}
For $x>0$ and $q>0$,
\begin{equation}\label{eq:mellin-recip}
 x^{-q}=\frac1{\Gamma(q)}\int_0^\infty
             s^{q-1}e^{-sx}\,\dd s.
\end{equation}
Thus Laplace order reverses the corresponding negative moments whenever the
benchmark moment is finite.  Apply \cref{thm:cofactor-lt} and Tonelli to get
\[
 \E\norm{C_n}^{-2q}\leq\E\Pi_n^{-q}.
\]
For $\gamma_r\sim\operatorname{Gamma}(r,1)$,
\[
       \E\gamma_r^{-q}
       =\frac{\Gamma(r-q)}{\Gamma(r)},\qquad 0<q<r.
\]
Independence gives \eqref{eq:cofactor-negative-q}; the smallest shape in
$\Pi_n$ is $3$.  The hafnian bound is identical, using
\eqref{eq:hafnian-lt}; now the smallest shape is $1$, so $q<1$.
\end{proof}

\begin{remark}
At $q=1$, \eqref{eq:cofactor-negative-q} reduces exactly to
\[
 \E\norm{C_n}^{-2}
 \leq\prod_{j=2}^n\frac1{2j-2}
 =\frac1{2^{n-1}(n-1)!},
\]
recovering \cref{cor:inverse-bound}.  Thus the direct recurrence in
\cref{sec:anticoncentration} is the $q=1$ shadow of the stronger product-Gamma
comparison.
\end{remark}

\section{Application to Gaussian boson sampling}\label{sec:gbs}

We now transfer \cref{thm:main} to the maximally squeezed GBS ensemble.  We
state only the random-matrix consequence needed here; the physical output
probabilities contain additional squeezing-dependent prefactors that do not
affect the present local statement.

Let $U$ be Haar distributed in $U(M)$.  Then $UU^{\mathsf T}$ has the circular
orthogonal ensemble (COE) distribution.  In the maximally squeezed setting,
the symmetric matrix entering collision-free GBS amplitudes is of this form,
and a fixed $N$-mode output pattern is represented by an $N\times N$ principal
block.  Let $A_{N,M}$ denote the upper-left $N\times N$ block of an $M\times M$
COE matrix.

Let $G_N^{\mathrm{sym}}$ be a symmetric complex Gaussian matrix whose
off-diagonal entries are independent $\CN(0,1)$ and whose diagonal entries
are independent $\CN(0,2)$.  The diagonal distribution is irrelevant to the
hafnian.  Shou, Miller, and Galitski prove that, for $N=o(\sqrt M)$,
\begin{equation}\label{eq:hiding-tv}
 d_{\mathrm{TV}}(\sqrt M A_{N,M},G_N^{\mathrm{sym}})
 =O\left(\frac{N}{\sqrt M}\right),
\end{equation}
and that if $N=o(M^{1/3})$, the corresponding matrix densities $f_{N,M}$ and
$g_N$ satisfy pointwise
\begin{equation}\label{eq:hiding-density}
 f_{N,M}(Z)
 \leq
 \left(1+O\left(\frac{N^3}{M}\right)\right)g_N(Z)
\end{equation}
for every symmetric complex matrix $Z$ \cite{ShouMillerGalitski2025}.

\begin{corollary}[COE hafnian small balls]\label{cor:coe}
Let $N=2n$.  If $n=o(M^{1/3})$, then uniformly in $z\in\C$ and
$\varepsilon>0$,
\begin{align}\label{eq:coe-pointwise}
 &\Prob\left(
  \left|M^{n/2}\haf(A_{2n,M})-z\right|
  \leq\varepsilon\sqrt{(2n-1)!!}
            \right)\nonumber\\
 &\qquad\leq
 \left(1+O\left(\frac{n^3}{M}\right)\right)
 B_n\varepsilon^2.
\end{align}
If $n=o(\sqrt M)$, then
\begin{align}\label{eq:coe-tv}
 &\Prob\left(
  \left|M^{n/2}\haf(A_{2n,M})-z\right|
  \leq\varepsilon\sqrt{(2n-1)!!}
            \right)\nonumber\\
 &\qquad\leq
 B_n\varepsilon^2+O\left(\frac n{\sqrt M}\right).
\end{align}
\end{corollary}

\begin{proof}
The hafnian of a $2n\times2n$ matrix is homogeneous of degree $n$, so
\[
       \haf(\sqrt M A_{2n,M})
       =M^{n/2}\haf(A_{2n,M}).
\]
For \eqref{eq:coe-pointwise}, integrate the density comparison
\eqref{eq:hiding-density} over the measurable set
\[
 \left\{Z:
   |\haf(Z)-z|\leq
   \varepsilon\sqrt{(2n-1)!!}
 \right\}
\]
and apply \cref{thm:main}.  For \eqref{eq:coe-tv}, total variation cannot
increase under a measurable map.  Push both matrix laws forward by the
hafnian map, use \eqref{eq:hiding-tv}, and again apply \cref{thm:main}.
\end{proof}

The pointwise hiding estimate therefore preserves the pure $\varepsilon^2$
local behavior in its range, while the total-variation statement gives the
same bound with an additive approximation error in the larger range.  This
should be interpreted as a probabilistic input to the GBS hardness program,
not as a complete hardness theorem.  Exact or approximate average-case
hardness for the corresponding random hafnians is a separate
complexity-theoretic ingredient, just as in the original boson-sampling
framework \cite{AaronsonArkhipov2013,DeshpandeEtAl2022}.

\section{Discussion}\label{sec:discussion}

The proof of \cref{thm:main} is driven by three structural identities.  The
one-vertex expansion
\[
              H_n=g^{\mathsf T}C_n
\]
turns a random hafnian into a one-dimensional Gaussian with random variance.
The two-vertex expansion
\[
 \sum_jw_j(C_n)_j
 =x^{\mathsf T}Ty+w_2x^{\mathsf T}\ell+w_1y^{\mathsf T}\ell+aq
\]
reduces a linear combination of cofactors to a Gaussian bilinear expression
plus one shared-edge term.  Finally,
\[
       \E_a e^{i\operatorname{Re}(aq)}=e^{-|q|^2/4}\ge0
\]
turns the apparent obstruction created by symmetry into a positive weight.
Once these identities are in place, the remaining argument consists of
Gaussian integration, H\"older interpolation, and Gamma calculus.

For the headline anticoncentration theorem, the cleanest consequence is the
inverse-moment recurrence $A_{n+1}\le A_n/(2n)$.  The same compression,
however, retains enough information to give the full product-Gamma
Laplace-transform domination of \cref{thm:cofactor-lt}.  The fractional
negative moments in \cref{cor:fractional-moments} illustrate that this
comparison is stronger than the single density estimate at the origin.
Further consequences of this stochastic comparison, and analogous local
bounds for other correlated hafnian ensembles, remain natural directions for
study.

\bibliographystyle{plain}
\bibliography{references}
\end{document}